\documentclass[a4paper]{amsart}

\usepackage{epic, eepic, amsfonts, latexsym, amssymb, graphicx,
multicol, mathrsfs, color, amscd, verbatim, paralist,
xspace, url, euscript, stmaryrd,  amsmath, enumitem,
bbold, multirow, tikz, mathtools}
\usepackage[all,pdf,cmtip]{xy}

\usepackage[colorlinks, linkcolor=blue, citecolor=magenta, urlcolor=cyan]{hyperref}

\usepackage{chngcntr}
\usepackage{apptools}

\def\hat{\widehat}
\renewcommand\bar{\overline}

\def\RR{{\mathbb R}}
\def\CC{{\mathbb C}}

\def\hat{\widehat}

\def\GL{\mathop{\rm GL}\nolimits}

\def\SO{\mathop{\rm SO}\nolimits}

\def\Aut{\mathop{\rm Aut}\nolimits}
\def\Im{\mathop{\rm Im}\nolimits}

\def\ad{\mathop{\rm ad}\nolimits}

\def\Aff{\mathop{\rm Aff}\nolimits}

\def\tr{\mathop{\rm tr}\nolimits}

\def\qed{{\hfill $\Box$}}
\newtheorem{theorem}{THEOREM}[section]

\theoremstyle{definition}

\newtheorem{lemma}[theorem]{Lemma}

\theoremstyle{remark}
\newtheorem{remark}[theorem]{Remark}

\makeatletter
\def\blfootnote{\xdef\@thefnmark{}\@footnotetext}
\makeatother

\begin{document}

\title[Manifolds with automorphism group of subcritical dimension]{Homogeneous Kobayashi-hyperbolic manifolds\\ with automorphism group of\\ subcritical dimension}
\author[Isaev]{Alexander Isaev}

\address{Mathematical Sciences Institute\\
Australian National University\\
Canberra, Acton, ACT 2601, Australia}
\email{alexander.isaev@anu.edu.au}

\maketitle

\thispagestyle{empty}

\pagestyle{myheadings}

\begin{abstract}We determine all connected homogeneous Kobayashi-hyperbolic manifolds of dimension $n\ge 2$ whose holomorphic automorphism group has dimension $n^2-3$. This result complements existing classifications for automorphism group dimension $n^2-2$ (which is in some sense critical) and greater.
\end{abstract}

\section{Introduction}\label{intro}
\setcounter{equation}{0}

Recall that a connected complex manifold $M$ is said to be \emph{Kobayashi-hyperbolic} if the Kobayashi pseudodistance $K_M$ on $M$ is in fact a distance, i.e., for $p,q\in M$ the identity $K_M(p,q)=0$ implies $p=q$. For instance, any bounded domain in a finite-dimensional complex vector space is Kobayashi-hyperbolic. Such manifolds are of interest in complex analysis and geometry as they enjoy a variety of nice properties (see \cite{K1}, \cite{K2} for details). In particular, if $M$ is Kobayashi-hyperbolic, the group $\Aut(M)$ of its holomorphic automorphisms is a (real) Lie group in the compact-open topology (see \cite[Chapter V, Theorem 2.1]{K1}). This follows, for example, from the fact that the action of $\Aut(M)$ on $M$ is proper, which implies that $\Aut(M)$ is locally compact hence a Lie transformation group (see the survey paper \cite{I5} for details). 

Let $n:=\dim_{\CC}M$ and assume that $n\ge 2$. Set $d(M):=\dim\Aut(M)$. It is a classical result that $d(M)\le n^2+2n$ with equality attained if and only if $M$ is biholomorphic to the unit ball $B^n\subset\CC^n$ (see \cite[Chapter V, Theorem 2.6]{K1}). In \cite{I1}, \cite{I2}, \cite{I4}, \cite{IK} we determined all Kobayashi-hyperbolic manifolds satisfying $n^2-1\le d(M)< n^2+2n$. Our classification has turned out to be quite useful in applications (see, e.g., \cite{V}), and one would like to extend it to automorphism group dimensions less than $n^2-1$. However, the value $n^2-2$ is critical in the sense that one cannot hope to obtain a full explicit description of Kobayashi-hyperbolic manifolds for $d(M)=n^2-2$ and all $n\ge 2$. Indeed, a generic Reinhardt domain in $\CC^2$ has a 2-dimensional automorphism group, so no reasonable classification exists for $n=2$ (see \cite[pp.~6--7]{I3} for a precise argument). Furthermore, an explicit classification for $n\ge 3$ also appears to be out of reach since even the easier case $d(M)=n^2-1$ is already rather complicated (see \cite{I1}, \cite{I4}). 

Nevertheless, some hope remains in the situation when $M$ is \emph{homogeneous}, i.e., when the action of $\Aut(M)$ on $M$ is transitive. Homogeneous manifolds are of general interest in geometry, and in \cite[Theorem 1.1]{I6} we classified all homogeneous Kobayashi-hyperbolic manifolds with automorphism group of critical dimension $n^2-2$. It is then natural to ask by how much further one can decrease $d(M)$ while still being able to produce reasonable explicit descriptions of the corresponding manifolds $M$. 

This paper is a follow-up to article \cite{I6}. Here we address the above question by making a step down from $n^2-2$ to subcritical dimension $n^2-3$. Specifically, we obtain:

\begin{theorem}\label{main}
Let $M$ be a homogeneous Kobayashi-hyperbolic manifold with $d(M)=n^2-3$. Then $n=4$ and $M$ is biholomorphic to $B^1\times T_3$, where
\begin{equation}
T_3:=\left\{(z_1,z_2,z_3)\in\CC^3:(\Im z_1)^2-(\Im z_2)^2-(\Im z_3)^2>0,\,\,\Im z_1>0\right\}\label{domaint3}
\end{equation}
is the symmetric bounded domain of type {\rm (}$\hbox{{\rm IV}}_3${\rm )} {\rm (}written in tube form{\rm )}.
\end{theorem}

\noindent In particular, it turns out that the product $B^1\times T_3$ is completely characterized by its automorphism group dimension among all homogeneous Kobayashi-hyperbolic manifolds. A similar effect is known for the ball $B^n$, the product $B^{n-1}\times B^1$, and the domain $T_3$ (see \cite[Theorem 2.2]{I3}). 

We note that making a further step down, i.e., producing an explicit classification for the case $d(M)=n^2-4$, appears to be much harder if not impossible, so it seems that Theorem \ref{main} edges quite closely to what one may hope to achieve in principle. Combining this theorem with the classical fact for $d(M)=n^2+2n$ mentioned earlier, \cite[Theorem 2.2]{I3} and \cite[Theorem 1.1]{I6}, we obtain:

\begin{theorem}\label{combined}
Let $M$ be a homogeneous Kobayashi-hyperbolic manifold satisfying $n^2-3\le d(M)\le n^2+2n$. Then $M$ is biholomorphic either to a suitable product of balls, or to a suitable symmetric bounded domain of type {\rm (IV)}, or to a suitable product of a ball and a symmetric bounded domain of type {\rm (IV)}. Specifically, the following products of balls are possible:
\begin{itemize}

\item[{\rm (i)}] $B^n$ {\rm (}here $d(M)=n^2+2n${\rm )},
\vspace{0.1cm}

\item[{\rm (ii)}] $B^{n-1}\times B^1$ {\rm (}here $d(M)=n^2+2${\rm )},
\vspace{0.1cm}

\item[{\rm (iii)}] $B^1\times B^1\times B^1$ {\rm (}here $n=3$, $d(M)=9=n^2${\rm )},
\vspace{0.1cm}

\item[{\rm (iv)}] $B^2\times B^2$ {\rm (}here $n=4$, $d(M)=16=n^2${\rm )},
\vspace{0.1cm}

\item[{\rm (v)}] $B^2\times B^1\times B^1$ {\rm (}here $n=4$, $d(M)=14=n^2-2${\rm )},
\vspace{0.1cm}

\item[{\rm (vi)}] $B^3\times B^2$ {\rm (}here $n=5$, $d(M)=23=n^2-2${\rm )},  
\end{itemize}

\noindent the following symmetric bounded domains of type {\rm (IV)} {\rm (}written in tube form{\rm )} are possible:

\begin{itemize}

\item[{\rm (vii)}] the domain of type {\rm (}$\hbox{{\rm IV}}_3${\rm )}, i.e., the domain $T_3$ defined in {\rm (\ref{domaint3})} {\rm (}here $n=3$, $d(M)=10=n^2+1${\rm )},
\vspace{0.3cm}

\item[{\rm (viii)}] the domain of type {\rm (}$\hbox{{\rm IV}}_4${\rm )}
\begin{equation}
\begin{array}{l}
T_4:=\left\{(z_1,z_2,z_3,z_4)\in\CC^4:(\Im z_1)^2-(\Im z_2)^2-\right.\\
\vspace{-0.3cm}\\
\hspace{5.5cm}\left.(\Im z_3)^2-(\Im z_4)^2>0,\,\,\Im z_1>0\right\}
\end{array}\label{domaint4}
\end{equation} 
{\rm (}here $n=4$, $d(M)=15=n^2-1${\rm )},

\end{itemize}

\noindent and the following product of a ball and a symmetric bounded domain of type {\rm (IV)}   is possible:

\begin{itemize}

\item[{\rm (ix)}] $B^1\times T_3$ {\rm (}here $n=4$, $d(M)=13=n^2-3${\rm )}.

\end{itemize}

\end{theorem}

The proof of Theorem \ref{main} is given in Section \ref{proof} and, just as the proof of the main theorem of \cite{I6}, is based on reduction to the case of the so-called \emph{Siegel domains of the second kind} introduced by I. Pyatetskii-Shapiro at the end of the 1950s (see Section \ref{prelim} for the definition). Indeed, in the important paper \cite{VGP-S} it was shown that every homogeneous bounded domain in $\CC^n$ is biholomorphic to an affinely homogeneous Siegel domain of the second kind. Moreover, in \cite{N} this result was extended to arbitrary homogeneous Kobayashi-hyperbolic manifolds, which solved a problem posed in \cite[p.~127]{K1}. Theorem \ref{main} is then deduced from the description of the Lie algebra of the automorphism group of a Siegel domain of the second kind given in \cite{KMO}, \cite[Chapter V, \S 1--2]{S}.

{\bf Acknowledgement.} This work is supported by the Australian Research\linebreak Council.

\section{Siegel Domains of the Second Kind}\label{prelim}
\setcounter{equation}{0}

Here we define Siegel domains of the second kind and collect their properties as required for our proof of Theorem \ref{main} in the next section. What follows is an abridged version of the exposition given in \cite[Section 2]{I6}.

To start with, an open subset $\Omega\subset\RR^k$ is called an \emph{open convex cone} if it is closed with respect to taking linear combinations of its elements with positive coefficients. Such a cone $\Omega$ is called \emph{{\rm (}linearly{\rm )} homogeneous} if the group
$$
G(\Omega):=\{A\in\GL_k(\RR):A\Omega=\Omega\}
$$
of linear automorphisms of $\Omega$ acts transitively on it. Clearly, $G(\Omega)$ is a closed subgroup of $\GL_k(\RR)$, and we denote by ${\mathfrak g}(\Omega)\subset{\mathfrak{gl}}_k(\RR)$ its Lie algebra.

We will be interested in open convex cones not containing entire lines. For such cones the dimension of ${\mathfrak g}(\Omega)$ admits a useful estimate.

\begin{lemma}\label{ourlemma}\cite[Lemma 2.1]{I6}\it Let $\Omega\subset\RR^k$ be an open convex cone not containing a line. Then
\begin{equation}
\dim {\frak g}(\Omega)\le\displaystyle\frac{k^2}{2}-\frac{k}{2}+1.\label{conineq}
\end{equation}
\end{lemma}

We note that estimate (\ref{conineq}) is sharp as the right-hand side is the dimension of the group of linear automorphisms of the cone
$\{x\in\RR^k:x_1^2-x_2^2-\cdots-x_k^2>0,\,\,x_1>0\}$.

Next, let
$$
H:\CC^m\times\CC^m\to\CC^k
$$
be a Hermitian form on $\CC^m$ with values in $\CC^k$, where we assume that $H(w,w')$ is linear in $w'$ and anti-linear in $w$. For an open convex cone $\Omega\subset\RR^k$, the form $H$ is called \emph{$\Omega$-Hermitian} if $H(w,w)\in\overline{\Omega}\setminus\{0\}$ for all non-zero $w\in\CC^m$. Observe that if $\Omega$ contains no lines and $H$ is $\Omega$-Hermitian, then there exists a positive-definite linear combination of the components of $H$.

Now, a Siegel domain of the second kind in $\CC^n$ is an unbounded domain of the form
$$
S(\Omega,H):=\left\{(z,w)\in\CC^k\times\CC^{n-k}:\Im z-H(w,w)\in \Omega\right\}
$$
for some $1\le k\le n$, some open convex cone $\Omega\subset\RR^k$ not containing a line, and some $\Omega$-Hermitian form $H$ on $\CC^{n-k}$. For $k=n$ we have $H=0$, so in this case $S(\Omega,H)$ is the tube domain       
$$
\left\{z\in\CC^n:\Im z\in \Omega\right\}.
$$
Such tube domains are often called \emph{Siegel domains of the first kind}. At the other extreme, when $k=1$, the domain $S(\Omega,H)$ is linearly equivalent to
$$
\left\{(z,w)\in\CC\times\CC^{n-1}:\Im z-||w||^2>0\right\},
$$ 
which is an unbounded realization of the unit ball $B^n$ (see \cite[p.~31]{R}). In fact, any Siegel domain of the second kind is biholomorphic to a bounded domain (see \cite[pp.~23--24]{P-S}), hence is Kobayashi-hyperbolic.

Next, the holomorphic affine automorphisms of Siegel domains of the second kind are described as follows (see \cite[pp.~25-26]{P-S}):

\begin{theorem}\label{Siegelaffautom}
Any holomorphic affine automorphism of $S(\Omega,H)$ has the form
$$
\begin{array}{lll}
z&\mapsto & Az+a+2iH(b,Bw)+iH(b,b),\\
\vspace{-0.3cm}\\
w&\mapsto & Bw+b,
\end{array}
$$
with $a\in\RR^k$, $b\in\CC^{n-k}$, $A\in G(\Omega)$, $B\in\GL_{n-k}(\CC)$, where
\begin{equation}
AH(w,w')=H(Bw,Bw')\label{assoc}
\end{equation}
for all $w,w'\in\CC^{n-k}$.
\end{theorem}

A domain $S(\Omega,H)$ is called \emph{affinely homogeneous} if the group $\Aff(S(\Omega,H))$ of its holomorphic affine automorphisms acts on $S(\Omega,H)$ transitively. Denote by $G(\Omega,H)$ the subgroup of $G(\Omega)$ that consists of all transformations $A\in G(\Omega)$ as in Theorem \ref{Siegelaffautom}, namely, of all elements $A\in G(\Omega)$ for which there exists $B\in\GL_{n-k}(\CC)$ such that (\ref{assoc}) holds. By \cite[Lemma 1.1]{D}, the subgroup $G(\Omega,H)$ is closed in $G(\Omega)$. It is easy to deduce from Theorem \ref{Siegelaffautom} that if $S(\Omega,H)$ is affinely homogeneous, the action of $G(\Omega,H)$ (hence that of its identity component $G(\Omega,H)^{\circ}$) is transitive on $\Omega$ (see, e.g., \cite[proof of Theorem 8]{KMO}), so the cone $\Omega$ is homogeneous. Conversely, if $G(\Omega,H)$ acts on $\Omega$ transitively, the domain $S(\Omega,H)$ is affinely homogeneous.

As shown in \cite{VGP-S}, \cite{N}, every homogeneous Kobayashi-hyperbolic manifold is biholomorphic to an affinely homogeneous Siegel domain of the second kind. Such a realization is unique up to affine transformations; in general, if two Siegel domains of the second kind are biholomorphic to each other, they are also equivalent by means of a linear transformation of special form (see \cite[Theorem 11]{KMO}). The result of \cite{VGP-S}, \cite{N} is the basis of our proof of Theorem \ref{main} in the next section.

In addition, our proof relies on a description of the Lie algebra of the group $\Aut(S(\Omega,H))$ of an arbitrary Siegel domain of the second kind $S(\Omega,H)$. This algebra is isomorphic to the (real) Lie algebra of complete holomorphic vector fields on $S(\Omega,H)$, which we denote by ${\mathfrak g}(S(\Omega,H))$ or, when there is no fear of confusion, simply by ${\mathfrak g}$. The latter algebra has been extensively studied. In particular, we have (see \cite[Theorems 4 and 5]{KMO}):

\begin{theorem}\label{kmoalgebradescr}
The algebra ${\mathfrak g}={\mathfrak g}(S(\Omega,H))$ admits a grading
$$
{\mathfrak g}={\mathfrak g}_{-1}\oplus{\mathfrak g}_{-1/2}\oplus{\mathfrak g}_0\oplus{\mathfrak g}_{1/2}\oplus{\mathfrak g}_1,
$$
with ${\mathfrak g}_{\nu}$ being the eigenspace with eigenvalue $\nu$ of $\ad\partial$, where
$\displaystyle\partial:=z\cdot\frac{\partial}{\partial z}+\frac{1}{2}w\cdot\frac{\partial}{\partial w}$. 
Here
$$
\begin{array}{ll}
{\mathfrak g}_{-1}=\displaystyle\left\{a\cdot\frac{\partial}{\partial z}:a\in\RR^k\right\},&\dim {\mathfrak g}_{-1}=k,\\
\vspace{-0.1cm}\\
{\mathfrak g}_{-1/2}=\displaystyle\left\{2i H(b,w)\cdot\frac{\partial}{\partial z}+b\cdot\frac{\partial}{\partial w}:b\in\CC^{n-k}\right\},&\dim {\mathfrak g}_{-1/2}=2(n-k),
\end{array}
$$
and ${\mathfrak g}_0$ consists of all vector fields of the form
\begin{equation}
(Az)\cdot\frac{\partial}{\partial z}+(Bw)\cdot\frac{\partial}{\partial w},\label{g0}
\end{equation}
with $A\in{\mathfrak g}(\Omega)$, $B\in{\mathfrak{gl}}_{n-k}(\CC)$ and
\begin{equation}
AH(w,w')=H(Bw,w')+H(w,Bw')\label{assoc1}
\end{equation}
for all $w,w'\in\CC^{n-k}$. Furthermore, one has
\begin{equation} 
\dim {\mathfrak g}_{1/2}\le 2(n-k),\qquad \dim {\mathfrak g}_1\le k.\label{estimm}
\end{equation}
\end{theorem}

It is then clear that the matrices $A$ that appear in (\ref{g0}) form the Lie algebra of $G(\Omega,H)$ and that ${\mathfrak g}_{-1}\oplus{\mathfrak g}_{-1/2}\oplus{\mathfrak g}_0$ is isomorphic to the Lie algebra of the group $\Aff(S(\Omega,H))$ (compare conditions (\ref{assoc}) and (\ref{assoc1})). 

Following \cite{S}, for a pair of matrices $A,B$ satisfying (\ref{assoc1}) we say that $B$ is \emph{associated to $A$} (with respect to $H$). Let ${\mathcal L}$ be the (real) subspace of ${\mathfrak{gl}}_{n-k}(\CC)$ of all matrices associated to the zero matrix in ${\mathfrak g}(\Omega)$, i.e., matrices skew-Hermitian with respect to each component of $H$. Set $s:=\dim {\mathcal L}$. Then we have
\begin{equation}
\dim {\mathfrak g}_0=s+\dim G(\Omega,H)\le s+\dim {\mathfrak g}(\Omega).\label{estim1}
\end{equation}
By Theorem \ref{kmoalgebradescr} and inequality (\ref{estim1}) one obtains
\begin{equation}
d(S(\Omega,H))\le k+2(n-k)+s+\dim {\mathfrak g}(\Omega)+\dim {\mathfrak g}_{1/2}+ \dim {\mathfrak g}_1,\label{estim 8}
\end{equation}
which, combined with (\ref{estimm}), leads to
\begin{equation}
d(S(\Omega,H))\le 2k+4(n-k)+s+\dim {\mathfrak g}(\Omega).\label{estim2}
\end{equation}
Further, since there exists a positive-definite linear combination ${\mathbf H}$ of the components of the Hermitian form $H$, the subspace ${\mathcal L}$ lies in the Lie algebra of matrices skew-Hermitian with respect to ${\mathbf H}$, thus
\begin{equation}
s\le (n-k)^2.\label{ests}
\end{equation}
By (\ref{ests}), inequality (\ref{estim2}) yields
\begin{equation}
d(S(\Omega,H))\le 2k+4(n-k)+(n-k)^2+\dim {\mathfrak g}(\Omega).\label{estim3}
\end{equation}
Combining (\ref{estim3}) with (\ref{conineq}), we deduce the following useful upper bound:
\begin{equation}
d(S(\Omega,H))\le\displaystyle\frac{3k^2}{2}-\left(2n+\frac{5}{2}\right)k+n^2+4n+1.\label{estim4}
\end{equation}

Next, by \cite[Chapter V, Proposition 2.1]{S} the component ${\mathfrak g}_{1/2}$ of the Lie algebra ${\mathfrak g}={\mathfrak g}(S(\Omega,H))$ is described as follows:

\begin{theorem}\label{descrg1/2}
The subspace ${\mathfrak g}_{1/2}$ consists of all vector fields of the form
$$
2iH(\Phi(\bar z),w)\cdot\frac{\partial}{\partial z}+(\Phi(z)+c(w,w))\cdot\frac{\partial}{\partial w},
$$
where $\Phi:\CC^k\to\CC^{n-k}$ is a $\CC$-linear map such that for every ${\mathbf w}\in\CC^{n-k}$ one has
$$
\Phi_{{\mathbf w}}:=\left[x\mapsto\Im H({\mathbf w},\Phi(x)),\,\, x\in\RR^k\right]\in{\mathfrak g}(\Omega),\label{Phiw0}
$$
and $c:\CC^{n-k}\times\CC^{n-k}\to\CC^{n-k}$ is a symmetric $\CC$-bilinear form on $\CC^{n-k}$ with values in $\CC^{n-k}$ satisfying the condition
$$
H(w,c(w',w'))=2iH(\Phi(H(w',w)),w')\label{cond1}
$$
for all $w,w'\in\CC^{n-k}$. 
\end{theorem}

Further, by \cite[Chapter V, Proposition 2.2]{S}, the component ${\mathfrak g}_1$ of ${\mathfrak g}={\mathfrak g}(S(\Omega,H))$ admits the following description:

\begin{theorem}\label{descrg1}
The subspace ${\mathfrak g}_1$ consists of all vector fields of the form
$$
a(z,z)\cdot\frac{\partial}{\partial z}+b(z,w)\cdot\frac{\partial}{\partial w},
$$
where $a:\RR^k\times\RR^k\to\RR^k$ is a symmetric $\RR$-bilinear form on $\RR^k$ with values in $\RR^k$ {\rm (}which we extend to a symmetric $\CC$-bilinear form on $\CC^k$ with values in $\CC^k${\rm )} such that for every ${\mathbf x}\in\RR^k$ one has
$$
A_{{\mathbf x}}:=\left[x\mapsto a({\mathbf x},x),\,\,x\in\RR^k\right]\in{\mathfrak g}(\Omega),\label{idents1}
$$
and $b:\CC^k\times\CC^{n-k}\to\CC^{n-k}$ is a $\CC$-bilinear map such that, if for ${\mathbf x}\in\RR^k$ one sets
$$
B_{{\mathbf x}}:=\left[w\mapsto\frac{1}{2}b({\mathbf x},w),\,\,w\in\CC^{n-k}\right],\label{idents2}
$$
the following conditions are satisfied:
\begin{itemize}

\item[{\rm (i)}] $B_{{\mathbf x}}$ is associated to $A_{{\mathbf x}}$ and $\Im\tr B_{{\mathbf x}}=0$ for all ${\mathbf x}\in\RR^k$,
\vspace{0.1cm}

\item[{\rm (ii)}] for every pair ${\mathbf w},{\mathbf w}'\in\CC^{n-k}$ one has
$$
B_{{\mathbf w},{\mathbf w}'}:=\left[x\mapsto\Im H({\mathbf w'},b(x,{\mathbf w})),\,\,x\in\RR^k\right]\in{\mathfrak g}(\Omega),
$$

\item[{\rm (iii)}] $H(w,b(H(w',w''),w''))=H(b(H(w'',w),w'),w'')$ for all $w,w',w''\in\CC^{n-k}$.

\end{itemize}
\end{theorem}

Next, let us recall the well-known classification, up to linear equivalence, of homogeneous convex cones not containing lines in dimensions $k=2,3,4$ (see, e.g., \cite[pp.~38--41]{KT}), which will be also required for our proof of Theorem \ref{main}:

\begin{itemize}

\item [$k=2$:] 

\begin{itemize}

\item[]

$\Omega_1:=\left\{(x_1,x_2)\in\RR^2:x_1>0,\,\,x_2>0\right\}$, where the algebra ${\mathfrak g}(\Omega_1)$ consists of all diagonal matrices, hence $\dim {\mathfrak g}(\Omega_1)=2$,
\end{itemize}
\vspace{0.3cm}

\item [$k=3$:] 

\begin{itemize}

\item[(i)] $\Omega_2:=\left\{(x_1,x_2,x_3)\in\RR^3:x_1>0,\,\,x_2>0,\,\,x_3>0\right\}$, where the algebra ${\mathfrak g}(\Omega_2)$ consists of all diagonal matrices, hence $\dim {\mathfrak g}(\Omega_2)=3$,
\vspace{0.1cm}

\item[(ii)] $\Omega_3:=\left\{(x_1,x_2,x_3)\in\RR^3:x_1^2-x_2^2-x_3^2>0,\,\,x_1>0\right\}$, where one has ${\mathfrak g}(\Omega_3)={\mathfrak c}({\mathfrak{gl}}_3(\RR))\oplus{\mathfrak o}_{1,2}$, hence $\dim {\mathfrak g}(\Omega_3)=4$; here for any Lie algebra ${\mathfrak h}$ we denote by ${\mathfrak c}({\mathfrak h})$ its center,

\end{itemize}
\vspace{0.3cm}

\item [$k=4$:] 

\begin{itemize}

\item[(i)] $\Omega_4:=\left\{(x_1,x_2,x_3,x_4)\in\RR^4:x_1>0,\,\,x_2>0,\,\,x_3>0,\,\,x_4>0\right\}$, where the algebra ${\mathfrak g}(\Omega_4)$ consists of all diagonal matrices, hence we have $\dim {\mathfrak g}(\Omega_4)=4$,
\vspace{0.1cm}

\item[(ii)] $\Omega_5:=\left\{(x_1,x_2,x_3,x_4)\in\RR^4: x_1^2-x_2^2-x_3^2>0,\,\,x_1>0,\,\,x_4>0\right\}$, where the algebra ${\mathfrak g}(\Omega_5)=\left({\mathfrak c}({\mathfrak{gl}}_3(\RR))\oplus{\mathfrak o}_{1,2}\right)\oplus\RR$ consists of block-diagonal matrices with blocks of sizes $3\times 3$ and $1\times 1$ corresponding to the two summands, hence $\dim {\mathfrak g}(\Omega_5)=5$,
\vspace{0.1cm}

\item[(iii)] $\Omega_6:=\left\{(x_1,x_2,x_3,x_4)\in\RR^4: x_1^2-x_2^2-x_3^2-x_4^2>0,\,\,x_1>0\right\}$, where ${\mathfrak g}(\Omega_6)={\mathfrak c}({\mathfrak{gl}}_4(\RR))\oplus{\mathfrak o}_{1,3}$, hence $\dim {\mathfrak g}(\Omega_6)=7$.
\end{itemize}
\end{itemize}

Finally, recall that in \cite{C}, \'E. Cartan found all homogeneous bounded domains in $\CC^2$ and $\CC^3$. We will now state an extension of Cartan's theorem to the case of Kobayashi-hyperbolic manifolds. A short proof based on Siegel domains of the second kind is given in \cite[Theorem 2.6]{I6}.

\begin{theorem}\label{cartansclass}\hfill 
\begin{itemize}

\item[{\rm (1)}] Every homogeneous Kobayashi-hyperbolic manifold of dimension $2$ is biholomorphic to one of

\begin{itemize}

\item[{\rm(i)}] $B^2$,
\vspace{0.1cm}

\item[{\rm(ii)}] $B^1\times B^1$.
\end{itemize}
\vspace{0.3cm}
 
\item[{\rm (2)}] Every homogeneous Kobayashi-hyperbolic manifold of dimension $3$ is biholomorphic to one of

\begin{itemize}

\item[{\rm(i)}] $B^3$,
\vspace{0.1cm}

\item[{\rm(ii)}] $B^2\times B^1$,
\vspace{0.1cm}

\item[{\rm(iii)}] $B^1\times B^1\times B^1$,
\vspace{0.1cm}

\item[{\rm(iv)}] the tube domain $T_3$ defined in {\rm (\ref{domaint3})}.

\end{itemize}
\end{itemize}

\end{theorem}

\section{Proof of Theorem \ref{main}}\label{proof}
\setcounter{equation}{0}

By \cite{VGP-S}, \cite{N}, the manifold $M$ is biholomorphic to a Siegel domain of the second kind $S(\Omega,H)$. Since for each domain listed in Theorem \ref{cartansclass} the dimension of its automorphism group is greater than $n^2-3$, it follows that $n\ge 4$. Also, as $M$ is not biholomorphic to $B^n$, we have $k\ge 2$.

Next, the following lemma rules out a large number of the remaining possibilities.

\begin{lemma}\label{n5k3} \it For $n\ge 5$ one cannot have $k\ge 4$, and for $n\ge 6$ one cannot have $k=3$.
\end{lemma}

\begin{proof} To establish the lemma, we will show that for $n\ge 5$, $k\ge 4$, as well as for $n\ge 6$, $k=3$, the right-hand side of inequality (\ref{estim4}) is strictly less than $n^2-3$, i.e., that for such $n,k$ the following holds:
$$
\frac{3k^2}{2}-\left(2n+\frac{5}{2}\right)k+4n+4<0.
$$
In order to see this, let us study the quadratic function
$$
\varphi(t):=\frac{3t^2}{2}-\left(2n+\frac{5}{2}\right)t+4n+4.
$$
Its discriminant is
$$
{\mathcal D}:=4n^2-14n-\frac{71}{4},
$$
which is easily seen to be positive for $n\ge 5$. Then the zeroes of $\varphi$ are
$$
\begin{array}{l}
\displaystyle t_1:=\frac{2n+\frac{5}{2}-\sqrt{{\mathcal D}}}{3},\\
\vspace{-0.1cm}\\
\displaystyle t_2:=\frac{2n+\frac{5}{2}+\sqrt{{\mathcal D}}}{3}.
\end{array}
$$

To establish the lemma, it suffices to show that: (i) $t_2>n$ for $n\ge 5$, (ii) $t_1<4$ for $n\ge 5$, (iii) $t_1<3$ for $n\ge 6$. Indeed, the inequality $t_2>n$ means that  
$$
n-\frac{5}{2}< \sqrt{{\mathcal D}},
$$
or, equivalently, that  
$$
n^2-3n-8>0,
$$
which is straightforward to verify for $n\ge 5$. Next, the inequality $t_1<4$ means that
$$
2n-\frac{19}{2}< \sqrt{{\mathcal D}},
$$
or, equivalently, that
$$
n>\frac{9}{2},
$$
which clearly holds if $n\ge 5$.  Finally, the inequality $t_1<3$ means that
$$
2n-\frac{13}{2}< \sqrt{{\mathcal D}},
$$
or, equivalently, that
$$
n>5,
$$
which completes the proof.\end{proof}

By Lemma \ref{n5k3}, in order to establish the theorem, we need to consider the following four cases: (1) $k=2$, $n\ge 4$, (2) $k=3$, $n=4$, (3) $k=3$, $n=5$, (4) $k=4$, $n=4$. Our arguments in Cases (1), (2) and (4) will be similar to those in the corresponding situations considered in \cite{I6}, whereas Case (3) is new. 
\vspace{-0.1cm}\\

{\bf Case (1).} Suppose that $k=2$, $n\ge 4$. Here $H=(H_1,H_2)$ is a pair of Hermitian forms on $\CC^{n-2}$. After a linear change of $z$-variables, we may assume that $H_1$ is positive-definite. In this situation, by applying a linear change of $w$-variables, we can simultaneously diagonalize $H_1$, $H_2$ as
$$
H_1(w,w)=||w||^2,\,\,\, H_2(w,w)=\sum_{j=1}^{n-2}\lambda_j|w_j|^2.
$$
If all the eigenvalues of $H_2$ are equal, $S(\Omega,H)$ is linearly equivalent either to
$$ 
D_1:=\left\{(z,w)\in\CC^2\times\CC^{n-2}:\Im z_1-||w||^2>0,\,\,\Im z_2>0\right\},
$$
or to
$$ 
D_2:=\left\{(z,w)\in\CC^2\times\CC^{n-2}:\Im z_1-||w||^2>0,\,\,\Im z_2-||w||^2>0\right\}.
$$ 
The domain $D_1$ is biholomorphic to $B^{n-1}\times B^1$, hence $d(D_1)=n^2+2>n^2-3$, which shows that $S(\Omega,H)$ cannot be equivalent to $D_1$. To deal with $D_2$, let us compute the group $G(\Omega_1,(||w||^2,||w||^2))$. It is straightforward to see that
$$
G(\Omega_1,(||w||^2,||w||^2))=\left\{\left(\begin{array}{cc}
\rho & 0\\
0 & \rho
\end{array}
\right),\,\,
\left(\begin{array}{cc}
0 & \displaystyle\eta\\
\displaystyle\eta & 0
\end{array}
\right)\,\,\hbox{with}\,\,\rho,\,\eta>0\right\},
$$ 
and it follows that the action of $G(\Omega_1,(||w||^2,||w||^2))$ is not transitive on $\Omega_1$. This proves that $S(\Omega,H)$ cannot be equivalent to $D_2$ either. Therefore, $H_2$ has at least one  pair of distinct eigenvalues.

Next, as $\dim{\mathfrak g}(\Omega)=2$, inequality (\ref{estim2}) yields
\begin{equation}
s\ge n^2-4n-1.\label{estim5}
\end{equation} 
On the other hand, by (\ref{ests}), we have
$$
s\le n^2-4n+4.
$$
More precisely, $s$ is calculated as
\begin{equation}
s=n^2-4n+4-2m,\label{estim7}
\end{equation}
where $m\ge 1$ is the number of pairs of distinct eigenvalues of $H_2$. Indeed, if
$$
B=\left(B_{ij}\right),\,\,\, B_{ij}=-\overline{B_{ji}},\,\,\, i,j=1,\dots,n-2,
$$
is skew-symmetric with respect to $H_1$, the condition of skew-symmetricity with respect to $H_2$ is written as
$$
B_{ij}\lambda_i=-\overline{B}_{ji}\lambda_j,\,\,\, i,j=1,\dots,n-2,
$$
which leads to $B_{ij}=0$ if $\lambda_i\ne\lambda_j$. 

By (\ref{estim5}), (\ref{estim7}) it follows that $1\le m\le 2$, thus we have either $n=4$ and $\lambda_1\ne\lambda_2$ (here $m=1$, $s=2$), or $n=5$ and, upon permutation of $w$-variables, $\lambda_1\ne\lambda_2=\lambda_3$ (here $m=2$, $s=5$). We will now consider these two situations separately.
\vspace{0.1cm}

{\bf Case (1a).} Suppose that $n=4$, $\lambda_1\ne\lambda_2$. Here, after a linear change of variables the domain $S(\Omega,H)$ takes the form
$$
\begin{array}{ll}
D_3:=\left\{(z,w)\in\CC^2\times\CC^2:\Im z_1-(\alpha|w_1|^2+\beta|w_2|^2)>0,\right.\\
\vspace{-0.3cm}\\
\hspace{5cm}\left.\Im z_2-(\gamma|w_1|^2+\delta|w_2|^2)>0\right\},
\end{array}\label{domaind3}
$$
where $\alpha,\beta,\gamma,\delta\ge 0$ and
$$
\det\left(\begin{array}{ll}
\alpha & \beta\\
\gamma & \delta
\end{array}
\right)\ne 0.
$$
We may also assume that $\alpha>0$. If $\beta=\gamma=0$, the domain $D_3$ is biholomorphic to $B^2\times B^2$. Since $d(B^2\times B^2)=16>13=n^2-3$, we in fact have $\beta+\gamma>0$. Using Theorem \ref{descrg1/2}, we showed in \cite[Lemma 3.2]{I6} that in this case for ${\mathfrak g}={\mathfrak g}(D_3)$ one has ${\mathfrak g}_{1/2}=0$. By estimate (\ref{estim 8}) and the second inequality in (\ref{estimm}), we then see
\begin{equation}
d(D_3)\le 12<13=n^2-3\label{estimmmmm1}
\end{equation}
(recall that $s=2$). This proves that $S(\Omega,H)$ cannot in fact be equivalent to $D_3$, so Case (1a) contributes nothing to the classification of homogeneous Kobayashi-hyperbolic $n$-dimensional manifolds with automorphism group dimension $n^2-3$.

\begin{remark}\label{rem1}
As explained in \cite[Remark 3.3]{I6}, estimate (\ref{estimmmmm1}) can be improved to $d(D_3)\le 10$ by showing that for $\beta+\gamma>0$ the component ${\mathfrak g}_1$ of the algebra ${\mathfrak g}={\mathfrak g}(D_3)$ is also zero. This fact is obtained by using Theorem \ref{descrg1}.
\end{remark}
\vspace{0.1cm}

{\bf Case (1b).} Suppose that $n=5$ and $\lambda_1\ne\lambda_2=\lambda_3$. Here, after a linear change of variables the domain $S(\Omega,H)$ takes the form
$$
\begin{array}{ll}
D_4:=\left\{(z,w)\in\CC^2\times\CC^3:\Im z_1-(\alpha|w_1|^2+\beta|w_2|^2+\beta|w_3|^2)>0,\right.\\
\vspace{-0.3cm}\\
\hspace{4cm}\left.\Im z_2-(\gamma|w_1|^2+\delta|w_2|^2+\delta|w_3|^2)>0\right\},
\end{array}\label{domaind4}
$$
where $\alpha,\beta,\gamma,\delta\ge 0$ and
$$
\det\left(\begin{array}{ll}
\alpha & \beta\\
\gamma & \delta
\end{array}
\right)\ne 0.
$$
As before, we may also assume that $\alpha>0$. Then, if $\beta=\gamma=0$, the domain $D_4$ is biholomorphic to $B^3\times B^2$. Since $d(B^3\times B^2)=23>22=n^2-3$, we have $\beta+\gamma>0$. Using Theorem \ref{descrg1/2}, we showed in \cite[Lemma 3.4]{I6} that in this case for ${\mathfrak g}={\mathfrak g}(D_4)$ one has ${\mathfrak g}_{1/2}=0$. By (\ref{estim 8}) and the second inequality in (\ref{estimm}), we then estimate
\begin{equation}
d(D_4)\le 17<22=n^2-3\label{estimmmmm2}
\end{equation}
(recall that here $s=5$). This proves that $S(\Omega,H)$ cannot in fact be equivalent to $D_4$, so Case (1b) contributes nothing to the classification of homogeneous Kobayashi-hyperbolic $n$-dimensional manifolds with automorphism group dimension $n^2-3$ either.

\begin{remark}\label{rem2}
As explained in \cite[Remark 3.5]{I6}, estimate (\ref{estimmmmm2}) can be improved to $d(D_4)\le 15$ by showing that for $\beta+\gamma>0$ the component ${\mathfrak g}_1$ of the algebra ${\mathfrak g}={\mathfrak g}(D_4)$ is also zero. This fact is obtained by utilizing Theorem \ref{descrg1}.
\end{remark}
\vspace{0.1cm}

{\bf Case (2).} Suppose that $k=3$, $n=4$. Here $S(\Omega,H)$ is linearly equivalent either to
\begin{equation}
D_5:=\left\{(z,w)\in\times\CC^3\times\CC:\Im z-v|w|^2\in\Omega_2\right\},\label{domaind5}
\end{equation}
where $v=(v_1,v_2,v_3)$ is a non-zero vector in $\RR^3$ with non-negative entries, or to
\begin{equation}
D_6:=\left\{(z,w)\in\times\CC^3\times\CC:\Im z-v|w|^2\in\Omega_3\right\},\label{domaind6}
\end{equation}
where $v=(v_1,v_2,v_3)$ is a vector in $\RR^3$ satisfying $v_1^2\ge v_2^2+v_3^2$, $v_1>0$. We will consider these two cases separately.
\vspace{0.1cm}

{\bf Case (2a).} Assume that $S(\Omega,H)$ is equivalent to the domain $D_5$ defined in (\ref{domaind5}). If only one entry of $v$ is non-zero, $D_5$ is biholomorphic to $B^2\times B^1\times B^1$. Since $d(B^2\times B^1\times B^1)=14>13=n^2-3$, we see that in fact at least two entries of $v$ are non-zero. 

Consider the identity component $G(\Omega_2,v|w|^2)^{\circ}$ of the group $G(\Omega_2,v|w|^2)$. As $G(\Omega_2,v|w|^2)^{\circ}$ lies in the identity component $G(\Omega_2)^{\circ}$ of $G(\Omega_2)$, every element of $G(\Omega_2,v|w|^2)^{\circ}$ is a diagonal matrix
\begin{equation}
\left(\begin{array}{lll}
\lambda_1 & 0 & 0\\
0 & \lambda_2 & 0\\
0 & 0 & \lambda_3
\end{array}
\right),\quad\lambda_j>0,\,\,j=1,2,3,\label{diagmatrrrcces}
\end{equation}
for which $v$ is an eigenvector. Therefore, if all entries of $v$ are non-zero, then $G(\Omega_2,v|w|^2)^{\circ}$ consists of scalar matrices, and if exactly two entries of $v$, say $v_i$ and $v_j$, are non-zero, then $G(\Omega_2,v|w|^2)^{\circ}$ consists of matrices of the form (\ref{diagmatrrrcces}) with $\lambda_i=\lambda_j$. In either situation, the action of $G(\Omega_2,v|w|^2)^{\circ}$ on $\Omega_2$ is not transitive. This shows that $S(\Omega,H)$ cannot in fact be equivalent to $D_5$, so Case (2a) contributes nothing to the classification of homogeneous Kobayashi-hyperbolic $n$-dimensional manifolds with automorphism group dimension $n^2-3$.
\vspace{0.1cm}

{\bf Case (2b).} Assume now that $S(\Omega,H)$ is equivalent to the domain $D_6$ defined in (\ref{domaind6}). Suppose first that $v_1^2>v_2^2+v_3^2$, i.e., that $v\in\Omega_3$. As the vector $v$ is an eigenvector of every element of $G(\Omega_3,v|w|^2)$, it then follows that $G(\Omega_3,v|w|^2)$ does not act transitively on $\Omega_3$. This shows that in fact we have $v_1=\sqrt{v_2^2+v_3^2}\ne 0$, i.e., $v\in\partial\Omega_3\setminus\{0\}$. Further, as the group $G(\Omega_3)^{\circ}=\RR_{+}\times\SO(1,2)^{\circ}$ acts transitively on $\partial\Omega_3\setminus\{0\}$, we suppose from now on that $v=(1,1,0)$.

\begin{lemma}\label{gomega3h}\it For the Hermitian form ${\mathcal H}(w,w'):=(\bar ww',\bar ww',0)$ we have
$$
\dim G(\Omega_3,{\mathcal H})=3.
$$
\end{lemma}

\begin{proof} We will compute the dimension of the Lie algebra of $G(\Omega_3,{\mathcal H})$, which we momentarily denote by ${\mathfrak h}$. Clearly, ${\mathfrak h}$ consists of all elements of ${\mathfrak g}(\Omega_3)$ having $(1,1,0)$ as an eigenvector. Recall now that 
$$
{\mathfrak g}(\Omega_3)={\mathfrak c}({\mathfrak{gl}}_3(\RR))\oplus{\mathfrak o}_{1,2}=\left\{
\left(\begin{array}{lll}
\lambda & p & q\\
p & \lambda & r\\
q & -r & \lambda
\end{array}
\right),\,\,\,\lambda,p,q,r\in\RR\right\}.
$$
It then follows that
$$
{\mathfrak h}=\left\{\left(\begin{array}{lll}
\lambda & p & q\\
p & \lambda & q\\
q & -q & \lambda
\end{array}
\right),\,\,\,\lambda,p,q\in\RR\right\}.
$$
In particular, $\dim {\mathfrak h}=3$ as required.\end{proof}

By Lemma \ref{gomega3h} we see that for ${\mathfrak g}={\mathfrak g}(D_6)$ one has $\dim{\mathfrak g}_0=4$ (recall that $s=1$). Furthermore, using Theorem \ref{descrg1/2}, we showed in \cite[Lemma 3.6]{I6} that ${\mathfrak g}_{1/2}=0$. Combining these two facts with the second inequality in (\ref{estimm}), we obtain 
\begin{equation}
d(D_6)=\dim{\mathfrak g}_{-1}+\dim{\mathfrak g}_{-1/2}+\dim{\mathfrak g}_0+\dim{\mathfrak g}_1\le 12<13=n^2-3.\label{estimmmmm3}
\end{equation}
This proves that $S(\Omega,H)$ cannot in fact be equivalent to $D_6$, so Case (2b) contributes nothing to the classification of homogeneous Kobayashi-hyperbolic $n$-dimen\-sional manifolds with automorphism group dimension $n^2-3$.

\begin{remark}\label{remg1d6} As explained in \cite[Remark 3.7]{I6}, one can utilize Theorem \ref{descrg1} to show that $\dim{\mathfrak g}_1=1$. Therefore, we in fact have $d(D_6)=10$, which improves\linebreak estimate (\ref{estimmmmm3}).  
\end{remark}
\vspace{0.1cm}

{\bf Case (3).} Suppose that $k=3$, $n=5$. Here $S(\Omega,H)$ is linearly equivalent either to
\begin{equation}
D_7:=\left\{(z,w)\in\times\CC^3\times\CC^2:\Im z-{\mathcal H}(w,w)\in\Omega_2\right\},\label{domaind7}
\end{equation}
where ${\mathcal H}$ is an $\Omega_2$-Hermitian form, or to
\begin{equation}
D_8:=\left\{(z,w)\in\times\CC^3\times\CC^2:\Im z-{\mathcal H}(w,w)\in\Omega_3\right\},\label{domaind8}
\end{equation}
where ${\mathcal H}$ is an $\Omega_3$-Hermitian form. We will consider these two cases separately.
\vspace{0.1cm}

{\bf Case (3a).} First, we will show that $S(\Omega,H)$ cannot be equivalent to the domain $D_7$ defined in (\ref{domaind7}). Indeed, recalling that $\dim{\mathfrak g}(\Omega_3)=3$, by estimate (\ref{estim3}) we have
$$
d(D_7)\le 21<22=n^2-3,
$$
which proves our claim. 
\vspace{0.1cm}

{\bf Case (3b).} Assume that $S(\Omega,H)$ is equivalent to the domain $D_8$ defined in (\ref{domaind8}). By (\ref{ests}) one has $s\le 4$. If $s<4$, by inequality (\ref{estim2}) we see
$$
d(D_8)\le 21<22=n^2-3.
$$
Therefore, we in fact have $s=4$.

Let ${\mathcal H}=({\mathcal H}_1,{\mathcal H}_2,{\mathcal H}_3)$ and ${\mathbf H}$ be a positive-definite linear combination of ${\mathcal H}_1$, ${\mathcal H}_2$, ${\mathcal H}_3$. By applying a linear change of $w$-variables, we can diagonalize ${\mathbf H}$ as\linebreak ${\mathbf H}(w,w)=||w||^2$. Since $s=4$, in these coordinates the subspace ${\mathcal L}$ coincides with the algebra ${\mathfrak u}_2$ of skew-Hermitian matrices of size $2\times 2$ (recall from Section \ref{prelim} that $s=\dim{\mathcal L}$). 

\begin{lemma}\label{skewherm} \it
Let ${\mathtt H}(w,w')$ be a Hermitian form on $\CC^m$ with values in $\CC$ such that
\begin{equation}
{\mathtt H}(Aw,w')+{\mathtt H}(w,Aw')=0\label{commut}
\end{equation}
for all $w,w'\in\CC^m$ and all $A\in {\mathfrak u}_m$. Then ${\mathtt H}(w,w)=\lambda ||w||^2$ for some $\lambda\in\RR$ and all $w\in\CC^m$.
\end{lemma} 

\begin{proof} We may assume that $m>1$. Let ${\mathrm H}$ denote the matrix of the Hermitian form ${\mathtt H}$. Then for any $A\in {\mathfrak u}_m$ condition (\ref{commut}) is satisfied for all $w,w'\in\CC^m$ if and only if $A$ and ${\mathrm H}$ commute. Hence ${\mathrm H}$ commutes with every element of ${\mathfrak u}_m$. 

Fix $1\le i_0<j_0\le m$ and let $A=(A_{ij})$ be the matrix with entries
$$ 
A_{ij}=\left\{\begin{array}{rl}
a & \hbox{if $(i,j)=(i_0,j_0)$},\\
\vspace{-0.3cm}\\
-\bar a & \hbox{if $(i,j)=(j_0,i_0)$},\\
\vspace{-0.3cm}\\
0 & \hbox{otherwise}
\end{array}
\right.
$$
for $a\in\CC$. Then we have
$$
\begin{array}{l}
(A{\mathrm H})_{i_0i_0}= a{\mathrm H}_{j_0i_0},\,\, (A{\mathrm H})_{j_0j_0}= -\bar a{\mathrm H}_{i_0j_0}, \,\, (A{\mathrm H})_{i_0j_0}= a{\mathrm H}_{j_0j_0},\,\, (A{\mathrm H})_{j_0i_0}=-\bar a{\mathrm H}_{i_0i_0}, \\
\vspace{-0.3cm}\\
({\mathrm H}A)_{i_0i_0}=-\bar a {\mathrm H}_{i_0 j_0}, \,\, ({\mathrm H}A)_{j_0j_0}=a{\mathrm H}_{j_0i_0}, \,\, ({\mathrm H}A)_{i_0j_0}= a{\mathrm H}_{i_0i_0}, \,\, ({\mathrm H}A)_{j_0i_0}=-\bar a{\mathrm H}_{j_0j_0}. 
\end{array}
$$
Now, the fact that $A$ and ${\mathrm H}$ commute for every $a\in\CC$ yields
$$
{\mathrm H}_{i_0i_0}={\mathrm H}_{j_0j_0},\,\,{\mathrm H}_{i_0j_0}=0,
$$ 
which completes the proof.\end{proof}

By Lemma \ref{skewherm}, the $\CC$-valued Hermitian forms ${\mathcal H}_1$, ${\mathcal H}_2$, ${\mathcal H}_3$ are proportional to ${\mathbf H}$. This shows that ${\mathcal H}(w,w)=v||w||^2$, where $v=(v_1,v_2,v_3)$ is a vector in $\RR^3$ satisfying $v_1^2\ge v_2^2+v_3^2$, $v_1>0$. 

We will now proceed as in Case (2b). Indeed, observe first that $v$ is an eigenvector of every element of $G(\Omega_3,v||w||^2)$. Then, if $v_1^2>v_2^2+v_3^2$, it follows that $G(\Omega_3,v||w||^2)$ does not act transitively on $\Omega_3$. Therefore $v_1=\sqrt{v_2^2+v_3^2}\ne 0$, i.e., $v\in\partial\Omega_3\setminus\{0\}$. As the group $G(\Omega_3)^{\circ}=\RR_{+}\times\SO(1,2)^{\circ}$ acts transitively on $\partial\Omega_3\setminus\{0\}$, we can suppose that $v=(1,1,0)$, so ${\mathcal H}(w,w)=(||w||^2,||w||^2,0)$.

Next, as in Lemma \ref{gomega3h}, we obtain
$$
\dim G(\Omega_3,{\mathcal H})=3.
$$
We then see that for ${\mathfrak g}={\mathfrak g}(D_8)$ one has $\dim{\mathfrak g}_0=7$ (recall that $s=4$). Combining this fact with inequalities (\ref{estimm}), we estimate 
$$
d(D_8)=\dim{\mathfrak g}_{-1}+\dim{\mathfrak g}_{-1/2}+\dim{\mathfrak g}_0+\dim{\mathfrak g}_{1/2}+\dim{\mathfrak g}_1\le 21<22=n^2-3.
$$
This proves that $S(\Omega,H)$ cannot in fact be equivalent to $D_8$, so Case (3b) contributes nothing to the classification of homogeneous Kobayashi-hyperbolic $n$-dimen\-sional manifolds with automorphism group dimension $n^2-3$.
\vspace{0.1cm}

{\bf Case (4).} Suppose that $k=4$, $n=4$. In this case, after a linear change of variables $S(\Omega,H)$ turns into one of the domains
$$
\begin{array}{l}
\left\{z\in\CC^4: \Im z\in\Omega_4\right\},\\
\vspace{-0.1cm}\\
\left\{z\in\CC^4: \Im z\in\Omega_5\right\},\\
\vspace{-0.1cm}\\
\left\{z\in\CC^4: \Im z\in\Omega_6\right\}
\end{array}
$$
and therefore is biholomorphic either to $B^1\times B^1\times B^1\times B^1$, or to $B^1\times T_3$, or to $T_4$, where $T_3$ and $T_4$ are the tube domains defined in (\ref{domaint3}), (\ref{domaint4}). The dimensions of the automorphism groups of these domains are 12, 13, 15, respectively. Noting that $n^2-3=13$, we see that $S(\Omega,H)$ is biholomorphic to the product $B^1\times T_3$, so Case (4) only contributes $B^1\times T_3$ to the classification of homogeneous Kobayashi-hyperbolic $n$-dimensional manifolds with automorphism group dimension $n^2-3$. 

The proof of Theorem \ref{main} is now complete.\qed

\end{document}